\documentclass[11pt,a4]{amsart}
\usepackage{amsmath}
\usepackage{amssymb} 
\usepackage{amscd} 
\usepackage[dvipdf]{graphicx} 
\usepackage[dvips]{color}
\usepackage[all]{xy} 
\usepackage{xcolor}

\usepackage{hyperref}
\hypersetup{
	colorlinks=true,
	linktoc=all,
    	linkcolor={red!50!black},
   	citecolor={blue!50!black},
  	urlcolor={blue!80!black}
	} 

\definecolor{dg}{rgb}{0.1,0.4,0.1}
\renewcommand{\labelenumi}{ $(\arabic{enumi})$ }

\newtheorem{theorem}{Theorem}[section]
\newtheorem{lemma}[theorem]{Lemma}

\newtheorem{proposition}[theorem]{Proposition}
\newtheorem{corollary}[theorem]{Corollary}
\newtheorem{claim}[theorem]{Claim}

\newtheorem{conjecture}[theorem]{Conjecture}
\newtheorem{remark}[theorem]{Remark}
\newtheorem{question}[theorem]{Question}

\theoremstyle{definition}
\newtheorem{definition}[theorem]{Definition}
\newtheorem{example}[theorem]{Example}

\newtheorem{problem}[theorem]{Problem}

\newtheorem*{thm_persistent_[G,G]_noncyclic}{Theorem~\ref{persistent_[G,G]_noncyclic}}

\newtheorem*{thm_same_trivializing_set}{Theorem~\ref{same_trivializing_set}}

\numberwithin{equation}{section}
\numberwithin{figure}{section}
\numberwithin{table}{section}

\renewcommand{\(}{\textup{(}}
\renewcommand{\)}{\textup{)}}

\begin{document}

\title[Group theoretic perspective on Dehn fillings]{Group theoretic perspective on Dehn fillings: 
Property P conjecture and beyond}

\author[T.Ito]{Tetsuya Ito}
\address{Department of Mathematics, Kyoto University, Kyoto 606-8502, JAPAN}
\email{tetitoh@math.kyoto-u.ac.jp}
\thanks{The first named author has been partially supported by JSPS KAKENHI Grant Number 19K03490, 21H04428, 23K03110.}

\author[K. Motegi]{Kimihiko Motegi}
\address{Department of Mathematics, Nihon University, 
3-25-40 Sakurajosui, Setagaya-ku, 
Tokyo 156--8550, Japan}
\email{motegi.kimihiko@nihon-u.ac.jp}
\thanks{The second named author has been partially supported by JSPS KAKENHI Grant Number 25K07018, 21H04428, 23K03110, 23K20791.}

\author[M. Teragaito]{Masakazu Teragaito}
\address{Department of Mathematics Education, Hiroshima University, 
1-1-1 Kagamiyama, Higashi-Hiroshima, 739--8524, Japan}
\email{teragai@hiroshima-u.ac.jp}
\thanks{The third named author has been partially supported by JSPS KAKENHI Grant Number 20K03587, 25K07004.}

\dedicatory{}

\begin{abstract}    
The Property P Conjecture, which was settled by Kronheimer and Mrowka, asserts that every 3--manifold obtained by non-trivial Dehn surgery on a non-trivial knot is never simply connected. 
We propose new perspectives in studying Dehn filling from group theoretic point of view, which stem from several variation of the Property P conjecture.
\end{abstract}

\maketitle


\section{Introduction}
\label{Introduction}
Poincar\'e had initially conjectured that any closed $3$--manifold with trivial homology is homeomorphic to the $3$--sphere $S^3$ \cite{Poincare_homology_IMT}. 
However, in 1904 Poincar\'e constructed a remarkable $3$--manifold which has a trivial homology, but whose fundamental group is non-trivial. 
This $3$--manifold is now called the Poincar\'e homology $3$--sphere, 
and this example made him to update the above conjecture to the {\it real} Poincar\'e conjecture which asserts that any closed $3$--manifold with trivial fundamental group is homeomorphic to $S^3$. 

In 1910, in his celebrated paper \cite{Dehn_IMT} Dehn introduced, so called \textit{Dehn surgery} and made a foundation of infinite group theory. 
Using Dehn surgery he constructed infinitely many homology $3$--spheres; 
it should be noted that before Dehn's work, only Poincar\'e's example (Poincar\'e homology $3$--sphere) was known. 

We may imagine that there were some efforts to find a counterexample to the Poincar\'e conjecture using Dehn surgery. 
Later Bing and Martin \cite{BingMartin_IMT} proposed the Property P conjecture which asserts that 
every non-trivial Dehn surgery on a non-trivial knot never produces a simply connected $3$--manifold. 
Since then the Property P conjecture had been playing a leading position in $3$--dimensional topology, 
and finally Kronheimer and Mrowka \cite{KM_IMT} settled the conjecture affirmatively using gauge theory and symplectic topology. 

In this article we discuss new perspectives of a study of Dehn fillings from group theoretic viewpoint 
motivated by Property P. 

\section{Dehn filling and Property P conjecture}
\label{Dehn filling}

Let $K$ be a non-trivial knot in $S^3$ with the exterior $E(K)$. 
Dehn filling on $E(K)$ is a geometric procedure to produce closed $3$--manifolds $K(r)$ 
by attaching (filling) a solid torus to $E(K)$. 
There are infinitely many ways to attach the solid torus to $E(K)$. 
Each attachment can be specified by using ``slopes'' on $\partial E(K)$. 
 
By the loop theorem the inclusion map 
$i\colon \partial E(K) \to E(K)$ induces a monomorphism 
$i_* \colon \pi_1(\partial E(K)) \to \pi_1(E(K))$, 
thence we have a subgroup $i_*(\pi_1(\partial E(K))) \subset \pi_1(E(K))$. 
We denote $\pi_1(E(K))$ by $G(K)$, which is called the knot group of $K$, 
and denote $i_*(\pi_1(\partial E(K)))$ by $P(K)$, which is called the peripheral subgroup of $K$. 
A \textit{slope element} in $G(K)$ is a primitive element $\gamma$ in $P(K) \cong \mathbb{Z} \oplus \mathbb{Z}$, 
which is represented by an oriented simple closed curve in $\partial E(K)$. 
Let us denote the normal closure of $\gamma$ in $G(K)$ by $\langle\!\langle \gamma \rangle\!\rangle$.   
Using the standard meridian-longitude pair $(\mu, \lambda)$ of $K$,  
each slope element $\gamma$ is expressed as $\mu^p \lambda^q$ for some relatively prime integers $p, q$. 
As usual we use the term \textit{slope} to mean the isotopy class of an unoriented simple closed curve in $\partial E(K)$. 
Two slope elements $\gamma$ and its inverse $\gamma^{-1}$ represent the same slope (by forgetting their orientations), 
which is identified with $p/q \in \mathbb{Q} \cup \{\infty\, (=1/0)\}$, 
where the meridian corresponds to $1/0$. 
Since $\langle\!\langle \gamma \rangle\!\rangle = \langle\!\langle \gamma^{-1} \rangle\!\rangle$, 
we may denote them by $\langle\!\langle p/q \rangle\!\rangle$. 
Thus each slope defines the normal subgroup $\langle\!\langle p/q \rangle\!\rangle \subset G(K)$, 
which will be referred to as the \textit{normal closure of the slope $p/q$} for simplicity. 

Attach a solid torus $S^1 \times D^2$ to $E(K)$ along their boundaries so that 
$\{ * \} \times \partial D^2$ represents a slope $r$ to obtain a closed $3$--manifold $K(r)$. 
We call $K(r)$ the 3-manifold obtained by \textit{$r$--Dehn filling} on $E(K)$. 
($K(r)$ is also called  the 3-manifold obtained by \textit{$r$--Dehn surgery} on $K$.) 

\begin{figure}[htb]
\centering
\includegraphics[width=0.7\linewidth]{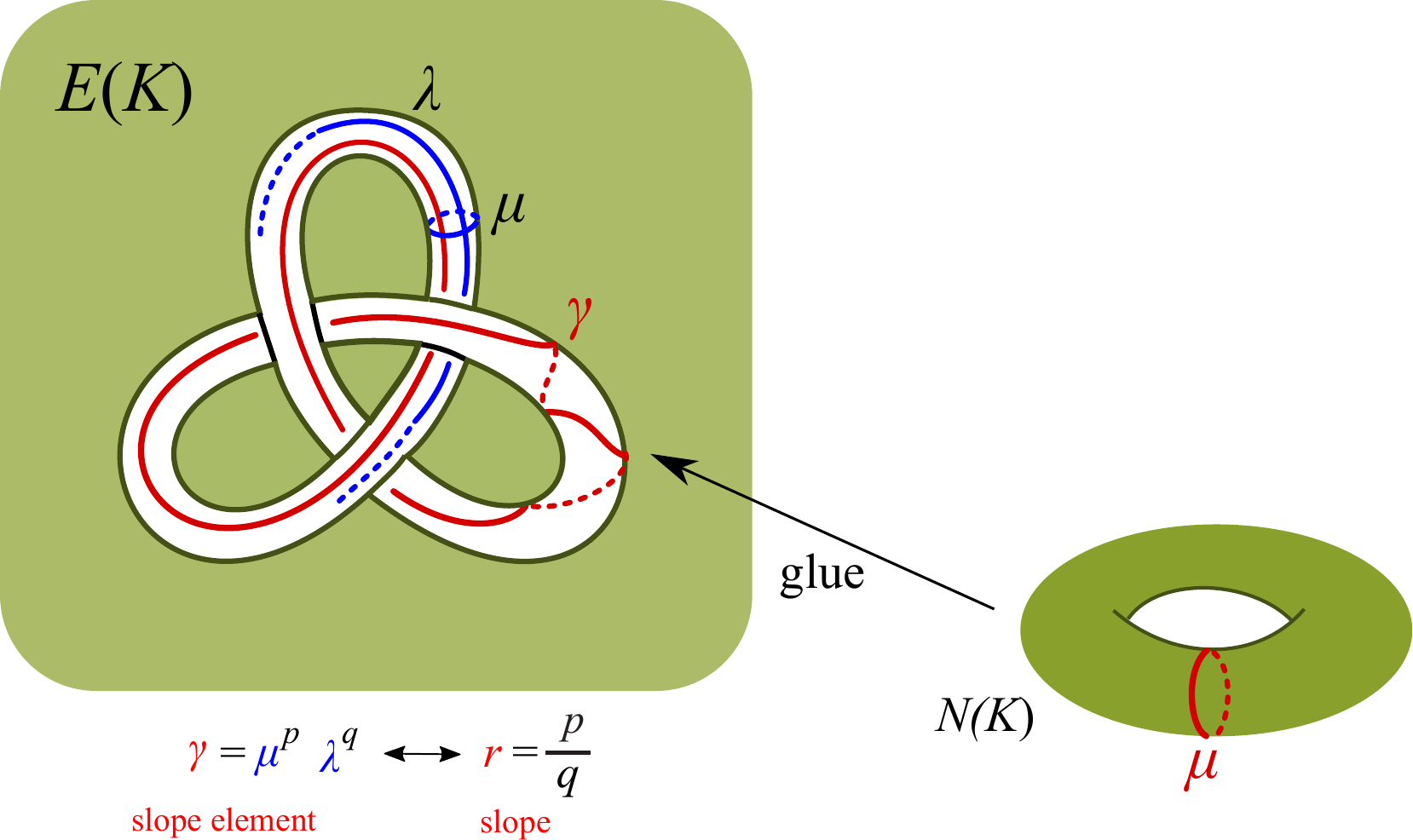}
\caption{Dehn filling} 
\label{filling}
\end{figure}

Then van Kampen's theorem tells us that
$\pi_1(K(p/q)) = G(K)/ \langle\!\langle \mu^p \lambda^q \rangle\!\rangle =  G(K)/\langle\!\langle p/q \rangle\!\rangle = G(K)/ \langle\!\langle r \rangle\!\rangle$, 
and we obtain a natural epimorphism
\[
p_r \colon G(K) \to G(K)/ \langle\!\langle r \rangle\!\rangle = \pi_1(K(r)), 
\]
which we call \textit{Dehn filling epimorphism}, or more specifically \textit{$r$--Dehn filling epimorphism}.  
The Dehn filling epimorphism gives us the following short exact sequence which relates $G(K),\ \langle\!\langle r \rangle\!\rangle$ and 
$\pi_1(K(r))$;
\[
1 \rightarrow \langle\!\langle r \rangle\!\rangle \rightarrow G(K) \xrightarrow{p_r} 
G(K) / \langle\!\langle r \rangle\!\rangle = \pi_1(K(r)) \rightarrow 1.
\]

Many people had an interest if one can obtain a homotopy $3$-sphere other than the $3$--sphere by Dehn surgery on knots. 
On the the hand, many knots turned out to have no non-trivial surgery yielding simply connected $3$--manifold, and  
Bing and Martin \cite{BingMartin_IMT} introduced the Property P. 

\medskip

\noindent
\textbf{Property P.}\quad 
{\it Every non-trivial Dehn filling of $E(K)$ yields a non-simply connected $3$--manifold.}

\medskip

Then they proposed the Property P conjecture below. 

\medskip

\begin{conjecture}[Property P conjecture]
Every non-trivial knot satisfies the Property P, 
i.e.
for every non-trivial knot $K$,  $\pi_1(K(r)) \ne 1$ for any $r \in \mathbb{Q}$.
\end{conjecture}

Note that the trivial knot does not satisfy the Property P. 
Indeed, $1/n$--Dehn filling of its exterior yields $S^3$. 

\medskip

\section{Brief history before the solution to the Property P conjecture}
The Property P conjecture provided driving motivation in Dehn surgery theory, 
and there are many research about this conjecture. 
We recall some distinguished results related to the Property P conjecture. 

The Cyclic Surgery Theorem due to Culler, Gordon, Luecke and Shalen \cite{CGLS_IMT} is one of the outstanding ones, 
which does hold in a more general setting, but here we restrict our attention to Dehn surgery on knots in $S^3$. 

\medskip

\begin{theorem}[Cyclic Surgery Theorem \cite{CGLS_IMT}]
\label{cyclic_surgery}
Let $K$ be a knot in $S^3$ other than a torus knot. 
Assume that both $K(p_1/q_1)$ and $K(p_2/q_2)$ have cyclic fundamental groups. 
Then the distance $\Delta(p_1/q_1, p_2/q_2)$ is less than or equal to one. 
Hence, there are at most three such slopes. 
\end{theorem}

In this theorem, the distance $\Delta(p_1/q_1, p_2/q_2)$ (introduced by \cite{G-Li_IMT}) is the minimal geometric intersection number between two slopes $p_1/q_1$ and $p_2/q_2$, 
which is given by $|p_1q_2 - p_2 q_1|$. 
Note that $\Delta$ does not satisfy the triangle inequality and it is not a usual distance function, 
but to describe exceptional slopes the notion of distance $\Delta$ is convenient and very useful. 
Since $K(1/0) = S^3$ has the trivial, and hence cyclic fundamental group, 
as a consequence of this theorem we have the following: 

\smallskip

$\bullet$ If $K$ is not a torus knot, then $K(p/q)$ has a cyclic fundamental group only when $|q| = 1$, i.e. the surgery must be integral. 

\smallskip

For torus knots the Property P conjecture is known to be true \cite{Moser_IMT}. 
Thus noting $H_1(K(p/q)) \cong \mathbb{Z}_{|p|}$, 
the Cyclic Surgery Theorem implies that 

\smallskip

$\bullet$ $K(p/q)$ is simply connected only when $p/q = \pm 1$. 

\medskip

Replacing the condition ``$\pi_1(K(r)) = \{ 1 \}$'' with the condition ``$K(r) = S^3$'', 
Gordon and Luecke proved the following result. 

\medskip

\begin{theorem}[\cite{GLu_S3_IMT}]
\label{S3}
Let $K$ be a non-trivial knot in $S^3$. 
Then $K(r) \not\cong S^3$ for all $r \in \mathbb{Q}$. 
\end{theorem}

This theorem says that a meridian, a boundary of a meridian disk of $N(K) \cong K \times D^2$, 
is uniquely determined by the exterior $S^3 - \mathrm{int}N(K)$. 
Hence, this theorem immediately implies a positive answer to 
the ``Knot Complement Conjecture'' proposed by Tietze \cite{Ti_IMT}. 

\medskip

\begin{theorem}[Knot complement theorem \cite{GLu_S3_IMT}]
\label{complement}
Knots are determined by their complements, 
i.e. if there exists an orientation preserving homeomorphism 
from $S^3 - K_1$ to $S^3 - K_2$, 
then there exists an orientation preserving homeomorphism of $S^3$ which sends $K_1$ to $K_2$. 
\end{theorem}

\medskip

Theorem~\ref{S3} is regarded as a specialization of the Property P, 
but its proof also provides the following important result as well. 

\medskip

$\bullet$ If $K(r)$ is reducible, meaning that it contains a $2$--sphere not bounding any $3$--ball, 
then it has a lens space $(\ne S^3, S^2 \times S^1)$ as a connected summand. 
Such a slope $r$ is called a {\it reducing surgery slope}.
As a consequence, any homology $3$--sphere obtained by Dehn surgery on a knot in $S^3$ is irreducible. 

\medskip

As we mentioned above, 
for any non-trivial knot $K$,  
$\pi_1(K(r)) =  \{ 1 \}$ only when $r = \pm 1$. 
Note also that for the mirror image $\overline{K}$ of $K$, 
there is an orientation reversing homeomorphism from $\overline{K}(-r)$ to $K(r)$. 
Thus to solve the Property P conjecture it is sufficient to show that $K(1)$ is not simply connected for a non-trivial knot $K$. 
Kronheimer and Mrowka \cite{KM_IMT} settled the Property P conjecture by 
proving an existence of a non-trivial representation from $\pi_1(K(1))$ to $\mathrm{SO}(3)$ for any non-trivial knot. 

\medskip

\noindent
\textbf{Property P Theorem \cite{KM_IMT}.}
{\it Let $K$ be a non-trivial knot in $S^3$. 
Then $\pi_1(K(r)) \ne 1$ for any $r \in \mathbb{Q}$. }

\medskip

\section{Property P and peripheral Magnus property}
\label{peripheral Magnus property} 
Recall that $\pi_1(K(r)) = G(K)/ \langle\!\langle r \rangle\!\rangle$. 
In particular, since $K(\infty) = S^3$, 
the knot group $G(K)$ is normally generated by the meridian, 
i.e. $G(K) = \langle\!\langle \infty \rangle\!\rangle$. 
Then Property P Theorem can be rephrased in the following form without appealing Dehn filling, 
purely in the language of the knot group. 

\medskip

\noindent
\textbf{Property P Theorem ($2^{\mathrm{nd}}$ formulation).}
{\it For any non-trivial knot, 
\[
\langle\!\langle r \rangle\!\rangle = \langle\!\langle \infty \rangle\!\rangle \ \textrm{if and only if}\  
r = \infty.
\] 
}

\medskip

Hence, $\langle\!\langle r \rangle\!\rangle$ is a proper subgroup of $G(K)$ for all $r \in \mathbb{Q}$. 

\medskip

This paraphrase leads us to extend the Property P Theorem ($2^{\mathrm{nd}}$ formulation) to all slopes $r$ in $\mathbb{Q} \cup \{ \infty \}$ 
in order to establish the one to one correspondence between 
slopes and their normal closures. 

\medskip

\begin{theorem}[Peripheral Magnus property \cite{IMT_Magnus_IMT}]
\label{peripheral_Magnus}
Let $K$ be a non-trivial knot and $r, r'$ slopes in $\mathbb{Q} \cup \{ \infty \}$. 
Then 
\[
\langle\!\langle r \rangle\!\rangle = \langle\!\langle r' \rangle\!\rangle\  
\textrm{if and only if}\  r = r'.
\]
\end{theorem}

\medskip

We should emphasize that 
even when $\langle\!\langle r \rangle\!\rangle \ne \langle\!\langle r' \rangle\!\rangle$, 
we may have isomorphic fundamental groups $\pi_1(K(r)) = G(K) / \langle\!\langle r \rangle\!\rangle$ and  
$\pi_1(K(r')) = G(K) /  \langle\!\langle r' \rangle\!\rangle$.  
That is, Theorem~\ref{peripheral_Magnus} does not imply $\pi_1(K(r)) = \pi_1(K(r'))$ if and only if  $r=r'$. 
Indeed, for the trefoil knot $T_{2, 3}$, 
there are pairs of distinct rational numbers $r, r'$ such that $K(r)$ is orientation reversingly homeomorphic to $K(r')$. 
For details, see \cite{Mathieu_IMT} and \cite[Proposition~3.1]{IMT_Magnus_IMT}.

\section{Dehn filling trivialization}
\label{function}
Recall that each $r$--Dehn filling induces an epimorphism 
$p_r \colon G(K) \to G(K) /  \langle\!\langle r \rangle\!\rangle = \pi_1(K(r))$, 
which trivializes elements in $\langle\!\langle r \rangle\!\rangle \subset G(K)$. 
We call such trivialization a {\em Dehn filling trivialization}. 

Dually, for a given non-trivial element of $G(K)$, 
we are intrigued by slopes $r$ for which $g$ becomes trivial after $r$--Dehn filling. 
Let us introduce the following function on $G(K)$. 

\medskip

\begin{definition}[Dehn filling trivializing slope set]
\label{D}
Let $K$ be a non-trivial knot in $S^3$. 
To each element $g \in G(K)$ assigning a subset 
\[
\mathcal{S}_K(g) 
=  \{ r \in \mathbb{Q} \mid p_r(g) = 1 \in \pi_1(K(r)) \}
=  \{ r \in \mathbb{Q} \mid g \in \langle\!\langle r \rangle\!\rangle \} \subset \mathbb{Q}, 
\]
we obtain a set valued function 
\[
\mathcal{S}_K \colon G(K) \to 2^{\mathbb{Q}}. 
\]
We call $\mathcal{S}_K(g)$ the {\it Dehn filling trivializing slope set} of $g$, 
or the {\it trivializing slope set} of $g$ for short. 
\end{definition}

\medskip

By definition, for the trivial element $1$, 
$\mathcal{S}_K(1) = \mathbb{Q}$. 

In general, a problem of deciding whether a given element (written by fixed generators as a word) represents the trivial element was formulated by Dehn \cite{Dehn_IMT,Dehn1911_IMT} as an algorithmic problem called the \textit{word problem}. 
Determination of $\mathcal{S}(g)$ requires a ``parametrized word problem'', and it is very hard. 

The function $\mathcal{S}_K \colon G(K) \to 2^{\mathbb{Q}}$ is far from injective. 
In fact, 
obviously, Dehn filling trivializing slope set satisfies 
$\mathcal{S}_K(\alpha^{-1} g \alpha) = \mathcal{S}_K(g)$ for any $\alpha$ and $g$ in $G(K)$. 

Furthermore, we have the following result. 
For notational simplicity, we denote $b^{-1} a b$ by $a^{b}$ 
for $a, b \in G(K)$.  

\medskip

\begin{theorem}
\label{same_trivializing_set}
Let $K$ be a non-trivial knot. 
For any non-trivial element $g \in G(K)$, we have 
\[
\mathcal{S}_K(g) = \mathcal{S}_K(g^{g^{\alpha}} g^{-2})\ \textrm{for any}\ \alpha \in G(K).
\]
\end{theorem}

We will prove Theorem~\ref{same_trivializing_set} in Section~\ref{section:Residual-finiteness} using residual finiteness of $3$--manifold groups. 
A similar result is obtained in \cite{IMT_realization_IMT} by completely different manner. 

On the other hand, since the knot group $G(K)$ is countable, 
Cantor's theorem says that the function $\mathcal{S}_K \colon G(K) \to 2^{\mathbb{Q}}$ is not surjective. 

In particular, for hyperbolic knots we have the following finiteness property. 

\medskip

\begin{theorem}[Finiteness theorem \cite{Osin_IMT,GM_IMT,IchiMT_IMT}]
\label{S_K_hyperbolic}
For any hyperbolic knot $K$,  
$\mathcal{S}_K(g)$ is always finite for any non-trivial element $g \in G(K)$. 
\end{theorem}

\medskip

\begin{remark}
\label{remark_finiteness}
\begin{enumerate}
\item
Although $\mathcal{S}_K(g)$ is always finite for every non-trivial element, 
there is no
universal bound. 
Indeed, for a given integer $N > 0$, 
we have a non-trivial element $g$ for which $|\mathcal{S}_K(g)| > N$ \cite{IMT_Magnus_IMT}. 

\item
Theorem~\ref{S_K_hyperbolic} does not hold for torus knots. 
Actually, 
for $(p, q)$--torus knot $T_{p, q}$, 
$\frac{m}{n}$--surgery on $T_{p, q}$, where $\frac{m}{n}$ satisfies $|pqn - m | = 1$ yields a lens space, 
and hence every element $g \in [G(T_{p, q}), G(T_{p, q})]$ becomes trivial in $\pi_1(T_{p, q}(m/n)) = \mathbb{Z}_m$. 
This means that $\mathcal{S}_{T_{p, q}}(g) \supset \{ pq \pm 1, \frac{2pq \pm 1}{2}, \dots, \frac{npq \pm 1}{n} , \dots \}$. 
\end{enumerate}

\end{remark}

Let us restate Property P theorem in the context of trivializing slope set. 

\medskip

\noindent
\textbf{Property P Theorem ($3^{\mathrm{rd}}$ formulation).}
{\it Let $K$ be a non-trivial knot. 
For a given non-trivial slope $r \in \mathbb{Q}$, 
we have an element $g \in G(K)$ such that $\mathcal{S}_K(g) \not\ni r$.}

\medskip

Then is it possible to select such an element $g$ so that
\begin{itemize}
\item
$\mathcal{S}_K(g) \ni s$ for a given $s \ne r$?

\medskip

\item
More challengingly, $\mathcal{S}_K(g) = \emptyset$?
\end{itemize}

\medskip

The first question leads us to the Separation Property below. 

\medskip

\noindent
\textbf{Separation Property.}
\begin{enumerate}
\renewcommand{\labelenumi}{(\roman{enumi})}
\item
{\it For given two slopes $r$ and $s$ in $\mathbb{Q}$, 
there exists an element $g \in G(K)$ such that 
$r \in \mathcal{S}_K(g)$ and $s \not\in \mathcal{S}_K(g)$.}

\medskip

\item 
{\it For any disjoint subsets $\mathcal{R}$ and $\mathcal{S}$ of $\mathbb{Q}$, 
there exists an element $g \in G(K)$ such that 
$\mathcal{R} \subset \mathcal{S}_K(g) \subset \mathbb{Q} -  \mathcal{S}$.}
\end{enumerate}

\medskip
Concerning (i), we prove Theorem~\ref{separation_two_slopes}. 
See also Examples~\ref{torus_knot_cyclic} and \ref{pretzel_non_separation}. 
In (ii) as we mentioned in Theorem~\ref{S_K_hyperbolic}, 
for any hyperbolic knot $K$ and any non-trivial element $g \in G(K)$,  
$\mathcal{S}_K(g)$ is a finite set, 
and hence it seems to be reasonable to restrict our attention to finite subsets $\mathcal{R}$.

In \cite{IMT_realization_IMT} we have shown the following result for hyperbolic knots. 
We say that a slope $r$ is a {\it Seifert surgery slope} if $K(r)$ is a Seifert fiber space. 

\medskip

\begin{theorem}[\cite{IMT_realization_IMT}]
\label{separation_hyperbolic}
Let $K$ be a hyperbolic knot. 
Let $\mathcal{R} = \{ r_1, \dots ,r_n\}$ and $\mathcal{S} = \{s_1, \dots , s_m\}$ 
be any finite, non-empty subsets of $\mathbb{Q}$ such that $\mathcal{R} \cap \mathcal{S} = \emptyset$.  
Assume that $\mathcal{S}$ does not contain a Seifert surgery slope. 
Then there exists an element $g \in [G(K ), G(K )] \subset G(K )$ such that $\mathcal{R} \subset  \mathcal{S}_K(g) \subset \mathbb{Q}- \mathcal{S}$. 
\end{theorem}

\medskip

In this article, 
we will generalize this to the following result which includes satellite knots. 

\medskip

\begin{theorem}[Separation theorem]
\label{separation}
Let $K$ be a non-trivial knot which is not a torus knot.  
Let $\mathcal{R} = \{ r_1, \dots, r_n \}$ and $\mathcal{S} = \{ s_1, \dots, s_m \}$ be any finite subsets of $\mathbb{Q}$ such that 
$\mathcal{R} \cap \mathcal{S} = \emptyset$. 
Assume that $\mathcal{S}$ does not contain a Seifert surgery slope. 
Then there exists an element $g \in [G(K), G(K) ]\subset G(K)$ such that $\mathcal{R} \subset \mathcal{S}_K(g) \subset \mathbb{Q} - \mathcal{S}$. 
\end{theorem}

\medskip

As we observe in Section~\ref{example}, 
Theorem~\ref{separation} does not hold unconditionally. 

\medskip

Let us turn to the second question. 
Since $G(K) = \langle\!\langle \mu \rangle\!\rangle$, 
$\pi_1(K(r))$ is also normally generated by a single element $p_r(\mu)$, 
i.e. $\pi_1(K(r)) = \langle\!\langle p_r(\mu) \rangle\!\rangle$. 
Hence, Property P implies that $p_r(\mu) \ne 1$ for all slopes $r \in \mathbb{Q}$. 

Thus we obtain an equivalent statement of Property P Theorem below, 
which also determines $\mathcal{S}(g)$ in the case where $g$ is a meridian. 

\medskip

\noindent
\textbf{Property P Theorem ($4^{\mathrm{th}}$ formulation).}
{\textit{For any non-trivial knot $K$, 
$\mathcal{S}_K(\mu) = \emptyset$.}

\medskip

This then motivates us to find further elements with $\mathcal{S}_K(g) = \emptyset$. 
For convenience we call such an element a {\em persistent element}. 
In the forthcoming paper \cite{IMT_ubiquitous_IMT} we will discuss persistent elements in details.

Property P Theorem  ($4^{\mathrm{th}}$ formulation) says that the empty set can be realized by $\mathcal{S}_K(\mu)$ for any non-trivial knot $K$. 
Then Separation Property and Property P Theorem  ($4^{\mathrm{th}}$ formulation) inspire to the following Realization Property. 

\medskip

\noindent
\textbf{Realization Property.}\quad 
{\it 
Every finite subset $\mathcal{R} \subset \mathbb{Q}$ can be realized by $\mathcal{S}_K(g)$ for some element $g \in G(K)$.}

\medskip

As we mentioned above, even if $\mathcal{R}$ is finite, 
we may have two disjoint family of slopes $\mathcal{R}$ and $\mathcal{S}$ for which there is no element $g \in G(K)$ such that 
$\mathcal{R} \subset \mathcal{S}_K(g) \subset \mathbb{Q} - \mathcal{S}$. 
This means that $\mathcal{R}$ cannot be realized by $\mathcal{S}_K(g)$ for any element $g$. 

\medskip

We discussed the Realization Problem for hyperbolic knots in \cite{IMT_realization_IMT} and established the Realization Property subject to an additional condition. 
Theorem~\ref{separation_hyperbolic} is the key step for its proof. 
For simplicity we state the simplest case. 

\medskip

\begin{theorem}[Realization theorem \cite{IMT_realization_IMT}]
\label{realization}
Let $K$ be a hyperbolic knot without exceptional \(i.e. non-hyperbolic\) surgery. 
Let $\mathcal{R} = \{ r_1, \ldots, r_n\}$ be any finite \(possibly empty\) subset of $\mathbb{Q}$. 
Then there exists an element $g \in G(K)$ such that 
$\mathcal{S}_K(g) = \mathcal{R}$.
\end{theorem}

\medskip
\begin{example}
Let $K_1$ be the hyperbolic knot $6_3$ depicted by Figure~\ref{realization_examples} $($Left$)$. 
Following \cite[Theorem~1.1]{BrittenhamWu_IMT} it is the simplest hyperbolic knot without exceptional surgery with respect to the crossing numbers.  
Then every finite (possibly empty) subset $\mathcal{R} \subset \mathbb{Q}$ is realized as 
$\mathcal{S}_{K_1}(g)$ for some element $g \in G(K_1)$.
\end{example}

\medskip

\begin{theorem}[\cite{IMT_realization_IMT}]
\label{realization_exceptional}
Let $K$ be a hyperbolic knot and $\mathcal{R}$ the set of exceptional surgery slopes. 
Then there exists an element $g \in G(K)$ such that 
$\mathcal{S}_K(g) = \mathcal{R}$.
\end{theorem}

\medskip 

\begin{example}
Let $K_2$ be the $(-2, 3, 7)$--pretzel knot depicted by Figure~\ref{realization_examples} $($Right$)$. 
The set of exceptional surgery slopes is known to be $\{ 16, 17, 18, \frac{37}{2}, 19, 20 \}$. 
Then we have an element  $g \in G(K_2)$ with 
$\mathcal{S}_{K_2}(g) = \{ 16, 17, 18, \frac{37}{2}, 19, 20 \}$. 
\end{example}

\begin{figure}[htb]
\centering
\includegraphics[width=0.6\textwidth]{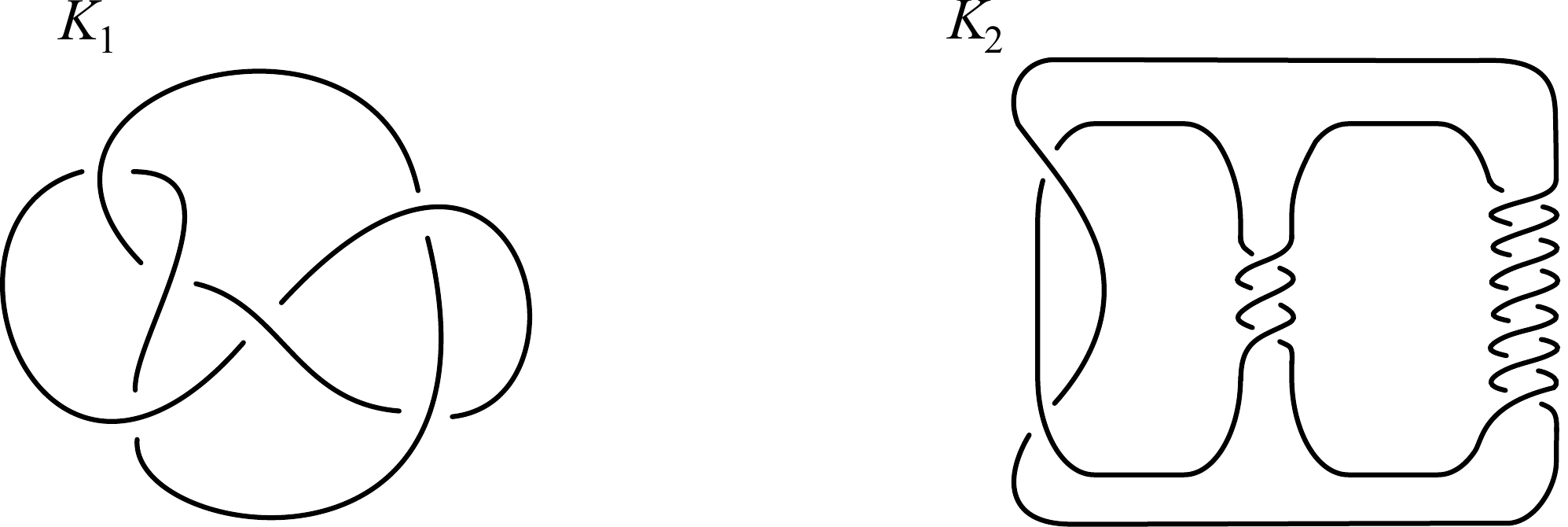}
\caption{The hyperbolic knot $K_1$ without exceptional surgery and the $(-2, 3, 7)$--pretzel knot $K_2$ } 
\label{realization_examples}
\end{figure}

\medskip
We close this section with the following future problem.

\medskip

\begin{problem}
\label{satellite}
Establish the Realization Property for non-hyperbolic knots under appropriate condition. 
\end{problem}

\section{Residual finiteness of $3$--manifold groups and Dehn filling trivializing slope sets}
\label{section:Residual-finiteness}

For any hyperbolic knot $K$, if $\mathcal{S}_K(g)$ is infinite, then Theorem~\ref{S_K_hyperbolic} shows that $g = 1$. 

\medskip

\begin{question}
\label{identical}
If $|\mathcal{S}_K(g)|$ is sufficiently large, 
then is an element $h$ satisfying $\mathcal{S}_K(h) = \mathcal{S}_K(g)$ conjugate to $g$ or some power of $g$?
\end{question}

\medskip

Using the residual finiteness of $3$--manifold groups, 
we will prove the following. 

\medskip

\begin{thm_same_trivializing_set}
Let $K$ be a non-trivial knot. 
For any non-trivial element $g \in G(K)$, we have 
\[
\mathcal{S}_K(g) = \mathcal{S}_K(g^{g^{\alpha}} g^{-2})\ \textrm{for any}\ \alpha \in G(K).
\]
\end{thm_same_trivializing_set}

\medskip

If $g \not\in [G(K), G(K)]$, 
then $g^{g^{\alpha}} g^{-2}$ is homologous to $g^{-1}$, 
hence it is not conjugate to $g^n$ ($n \ne -1$).

This result is quite general, 
but it is not so convenient to see that varying elements $\alpha$ we may actually obtain 
non-conjugate elements $g^{g^{\alpha}} g^{-2}$. 
See Proposition~\ref{conjugacy}. 
It should be interesting to compare Theorem~\ref{same_trivializing_set} with 
\cite[Theorem~1.12]{IMT_realization_IMT}.

The proof of Theorem~\ref{same_trivializing_set} requires both the residual finiteness of $3$--manifold groups and existence of non-residually finite one relator groups. 

We begin by recalling some definitions. 
A group $G$ is \textit{residually finite\/} if for each non-trivial element $g$ in $G$, 
there exists a normal subgroup of finite index not containing $g$.  
This is equivalent to say that for every $1\neq g \in G$ there exists a homomorphism 
$\varphi \colon G \rightarrow F$ to some finite group $F$ such that $\varphi(g)\neq 1$. 

Let $\mathcal{F}(G)$ be the set of elements of $G$ which is mapped to the trivial element for every finite group $F$ and every homomorphism $\varphi \colon G \rightarrow F$. Then, by definition, $G$ is residually finite if and only if $\mathcal{F}(G)=\{1\}$. 

For $3$--manifolds, we have:

\medskip

\begin{theorem}[\cite{Hem_residual_finite_IMT,Pe1_IMT,Pe2_IMT,Pe3_IMT}]
\label{residually_finite_3mfd}
The fundamental group of every compact $3$--manifold is residually finite.
\end{theorem}

\medskip

Let $F_n=\langle a_1,a_2,\ldots, a_n\rangle$ be the free group of rank $n  > 1$ generated by $a_1,a_2,\ldots, a_n$.
For an element $w \in F_n$, 
which we regard as a reduced word over $\{a_1^{\pm 1},\ldots, a_n^{\pm 1}\}$ as usual, 
let us write
 \[ G_w := F_n \slash \langle\!\langle w \rangle \! \rangle = \langle a_1,a_2,\ldots, a_n\ | \: w \rangle,
  \]
which is the one-relator group whose relator is $w$. 
The natural projection of $F_n$ to $G_w$ is denoted by 
\[
\pi_w\colon F_n \rightarrow G_w. 
\]

\medskip

\begin{theorem}
\label{residually-finite}
Assume that we have two non-trivial elements $v$ and $w$ in $F_n$ such that 
\[
\pi_w(v) \in \mathcal{F}(G_w)- \{ 1 \}.
\] 
Then for any homomorphism $\phi \colon F_n \to G(K)$, 
we have 
\[
\mathcal{S}_K(\phi(w)) \subset \mathcal{S}_K(\phi(v)).
\] 
\end{theorem}

\begin{proof}
If $\mathcal{S}_K(\phi(w)) = \emptyset$, 
then nothing to prove. 
So we assume $\mathcal{S}_K(\phi(w)) \ne \emptyset$. 
Let $s$ be a slope in $\mathcal{S}_K(\phi(w))$. 
Since $p_s(\phi(w))=1$, 
there is a homomorphism $\psi\colon G_w \to  \pi_1(K(s))$ that makes the following diagram commutative.
\[
\xymatrix{
F_n \ar[r]^-{\phi} \ar[d]_{\pi_w} &  G(K) \ar[d]^{p_s}\\
G_w=F_n \slash \langle\!\langle w \rangle \! \rangle \ar[r]^-{\psi} & \pi_1(K(s))
}
\]
Assume to the contrary that $s \not \in \mathcal{S}_K(\phi(v))$, so $p_s(\phi(v)) \neq 1$. 
By the residual finiteness of $\pi_1(K(s))$ there exists a homomorphism $\kappa\colon \pi_1(K(s)) \to F$ to a finite group $F$ such that $\kappa(p_s(\phi(v))) = \kappa(\psi(\pi_w(v))) \neq 1$. 
However, this implies that $\kappa \circ \psi\colon G_w \rightarrow F$ satisfies $\kappa \circ \psi(\pi_w(v)) \neq 1$. 
This contradicts the assumption that $\pi_w(v) \in \mathcal{F}(G_w) - \{1\}$. 
\end{proof}

To exploit Theorem~\ref{residually-finite}, we need a non-residually finite one-relator group $G_w$ and an element in $\mathcal{F}(G_w) - \{1\}$. Although the most famous and the simplest example of non-residually finite one-relator group is the Baumslag--Solitar group \cite{BaumslagSolitar_IMT}, here we use a group given by Baumslag--Miller--Troeger \cite{BMT_IMT}, because for this relator $w$, $\mathcal{F}(G_w)-\{1\}$ contains an element in a quite simple form.

\medskip

\begin{theorem}[\cite{BMT_IMT}]
\label{theorem:BMT_IMT}
Let $u, v \in F_n=\langle a_1,a_2,\ldots, a_n \rangle$ be elements of free group of rank $n >1$ such that $uv \neq vu$. 
Let us take
\[
w = v^{v^{u}}v^{-2} \in F_n, 
\]
and a one-relator group 
\[
G_{w} = \langle a_1,a_2,\ldots, a_n \mid  w \rangle. 
\]
Then $G_{w}$ is not residually finite, 
and for the natural projection
$\pi_{w} \colon F_n \rightarrow G_{w}$, 
\[
\pi_{w}(v) \in \mathcal{F}(G_{w}) - \{ 1 \}.
\] 
\end{theorem}

\medskip

Applying Theorem~\ref{residually-finite} in this special case, we get the following.

\medskip

\begin{corollary}
\label{same_trivialization_residually_finite}
Let $u, v \in F_n=\langle a_1,a_2,\ldots, a_n \rangle$ be elements of free group of rank $n >1$ such that $uv \neq vu$. 
Take $w = v^{v^u}v^{-2}$. 
Let $\phi \colon F_n \rightarrow G(K)$ be a homomorphism. 
Then for $g = \phi(v)$ and $h = \phi(w)$, 
we have   $\mathcal{S}_K(g) = \mathcal{S}_K(h)$.
\end{corollary}

\begin{proof}
By the choice of elements $u, v, w$, 
following Theorem~\ref{theorem:BMT_IMT} we have 
\[
\pi_w(v) \in \mathcal{F}(G_{w}) - \{ 1 \}
\]
for the natural projection $\pi_w \colon F_n \to G_w$.

Then apply Theorem~\ref{residually-finite} to see that 
\[
\mathcal{S}_K(h) = \mathcal{S}_K(\phi(w)) \subset  \mathcal{S}_K(\phi(v)) = \mathcal{S}_K(g).
\] 

Conversely, for a slope $s$, 
if $p_s(g)=p_s(\phi(v))=1$, 
then 
\[
p_s(h)= p_s(\phi(v^{v^u}v^{-2}))=1
\]
so $\mathcal{S}_K(g) \subset \mathcal{S}_K(h)$.
\end{proof}

\medskip

\begin{proof}[Proof of Theorem~\ref{same_trivializing_set}]
Take $F_2=\langle a_1, a_2 \rangle$, $v=a_1$ and $u =a_2$, 
and let $\phi \colon F_2  \rightarrow G(K)$ be a homomorphism given by 
$\phi(v)=g$ and $\phi(u)= \alpha$, where $\alpha \in G(K)$ is chosen arbitrarily.
Then by Corollary~\ref{same_trivialization_residually_finite}, 
$h = \phi(v^{v^u}v^{-2}) = g^{g^{\alpha}} g^{-2}$ satisfies 
$\mathcal{S}_K(h) = \mathcal{S}_K(g)$. 
\end{proof}

\medskip
This construction can be repeated. 
For every $\alpha_1, \alpha_2, \ldots \in G(K)$, 
let us define $g_i$ inductively as $g_{i+1} = g_{i}^{g_i^{\alpha_i}}g_{i}^{-2}$, where $g_0 = g$.  
Then for each $i$, $\mathcal{S}_K(g_i) = \mathcal{S}_K(g)$.  

\bigskip

Note that $\alpha$ can be taken arbitrarily, 
so we may expect to obtain mutually non-conjugate elements $g^{g^{\alpha}} g^{-2}$ by varying $\alpha$. 
In Proposition~\ref{conjugacy} below, 
we introduce conjugacy invariants $r(g, \alpha) \in \mathbb{C}$ for $g^{g^{\alpha}} g^{-2}$ (depending upon $g$ is non-peripheral or peripheral) which enjoy the following property: 
for a given element $g \in G(K)$, 
if $g^{g^{\alpha}} g^{-2}$ is conjugate to $g^{g^{\alpha'}} g^{-2}$, 
then $r(g, \alpha) = r(g, \alpha')$.

\medskip

We denote an element in
$\mathrm{PSL}(2, \mathbb{C}) = \mathrm{SL}(2, \mathbb{C})/ \{ \pm I \}$ 
by 
$\begin{bmatrix}
a & b \\
c & d
\end{bmatrix}$
to distinguish from 
a matrix 
$\begin{pmatrix}
a & b \\
c & d
\end{pmatrix}$ 
in $\mathrm{SL}(2, \mathbb{C})$.
Then 
$\mathrm{tr}\begin{bmatrix}
a & b \\
c & d
\end{bmatrix}$ 
is understood to be $\pm(a+ d)$.

\medskip

\begin{proposition}
\label{conjugacy}
Let $K$ be a hyperbolic knot and $g, \alpha$ non-trivial elements of $G(K)$. 
Then we have the following. 

\begin{enumerate}
\renewcommand{\labelenumi}{(\arabic{enumi})}
\item
Assume that $g$ is non-peripheral. 
Let $\rho \colon G(K) \to \mathrm{PSL}_2(\mathbb{C})$ be a holonomy representation with 
$\rho(g) =  \begin{bmatrix}
	\zeta & 0 \\[2pt]
	0 & \zeta^{-1}
\end{bmatrix}$ and 
$\rho(\alpha) =  \begin{bmatrix}
	x & y \\[2pt]
	z & u
\end{bmatrix}$.
Then 
$r(g, \alpha) = (\zeta - \zeta^{-1})^2  (xu)^2 - (\zeta - \zeta^{-1})(xu) -1$ \(up to sign\) is an invariant 
of conjugacy class 
of $g^{g^{\alpha}}g^{-2}$. 

\item
Assume that $g$ is peripheral. 
Let $\rho \colon G(K) \to \mathrm{PSL}_2(\mathbb{C})$ be a holonomy representation with 
$\rho(g) 
=  \begin{bmatrix}
	1 & \zeta \\[2pt]
	0 & 1
\end{bmatrix}$ 
and 
$\rho(\alpha) 
=  \begin{bmatrix}
	x & y \\[2pt]
	z & u
\end{bmatrix}$. 
Then 
$r(g, \alpha) = 2z^4 \zeta^4 +2$ \(up to sign\) is an invariant of conjugacy class 
of $g^{g^{\alpha}}g^{-2}$. 
In particular, 
if $2z^4 \zeta^4 +2 \ne \pm 2$, 
then $g^{g^{\alpha}}g^{-2}$ is not peripheral. 
\end{enumerate}

\end{proposition}

\begin{proof}
(1)\ 
Put $\gamma = g^{\alpha} = \alpha^{-1} g \alpha$. 

Write
$\rho(\gamma) = 
\begin{bmatrix}
	a & b \\[2pt]
	c & d
\end{bmatrix}$.  
Then 
\[
\rho(g^{g^{\alpha}} g^{-2}) = \rho(\gamma^{-1} g \gamma g^{-2}) 
= 
\begin{bmatrix}
	ad \zeta^{-1} - bc\zeta^{-3} & bd \zeta^3 - bd \zeta \\[2pt]
	-ac\zeta^{1} + ac \zeta^{-3} & -bc \zeta^3 + ad \zeta
\end{bmatrix}.
\]

So we have 
\[
\mathrm{tr}\rho(g^{g^{\alpha}} g^{-2}) 
= \pm\left( ad\left( (\zeta + \zeta^{-1}) - (\zeta^3 + \zeta^{-3})\right) + (\zeta^3 + \zeta^{-3})\right).
\]  
Note that $(\zeta + \zeta^{-1}) - (\zeta^3 + \zeta^{-3}) \ne 0$. 
Suppose for a contradiction that $\zeta + \zeta^{-1} = \zeta^3 + \zeta^{-3} = (\zeta + \zeta^{-1})(\zeta^2  -1 + \zeta^{-2})$.
Then $\zeta^2 + \zeta^{-2} = 2$, which shows that $\zeta^2 = 1$. 
$\rho(g) = I$, hence $g$ is trivial. 
This is a contradiction.
Therefore $ad$ is a conjugacy invariant of $g^{g^{\alpha}} g^{-2}$. 
To express $ad$ by $\rho(\alpha)$, let us write 
$\rho(\alpha) 
= 
\begin{bmatrix}
	x & y \\[2pt]
	z & u
\end{bmatrix}
$.

Then $\rho(\gamma) = \rho(\alpha^{-1} g \alpha) 
= 
\begin{bmatrix}
	xu\zeta - yz \zeta^{-1} & yu\zeta - yu\zeta^{-1} \\[2pt]
	-xz \zeta + xz \zeta^{-1} & -ya \zeta + xu \zeta^{-1}
\end{bmatrix}
=
\begin{bmatrix}
	a & b \\[2pt]
	c & d
\end{bmatrix}$.

So  we have 
\[
ad = (xu\zeta - yz\zeta^{-1})(xu\zeta^{-1} - yz\zeta) = 
-(\zeta - \zeta^{-1})^2 (xu)^2 + (\zeta - \zeta^{-1})(xu) +1. 
\] 

Note that $\zeta - \zeta^{-1} \ne 0$, 
for otherwise $\zeta^2 = 1$ and $\rho(g) = I$, and hence $g = 1$, a contradiction.
Since $\alpha$ is arbitrarily chosen, 
we expect there are infinitely many elements $\alpha \in G(K)$ such that 
$xu$, and hence $ad$ take infinitely many values. 

\medskip

(2)\ 
Put $\gamma = g^{\alpha} = \alpha^{-1} g \alpha$. 

Write
$\rho(\gamma) = 
\begin{bmatrix}
	a & b \\[2pt]
	c & d
\end{bmatrix}$ as above. 
Then 
\[
\rho(g^{g^{\alpha}} g^{-2}) = \rho(\gamma^{-1} g \gamma g^{-2}) 
= 
\begin{bmatrix}
	1+cd \zeta & - 2\zeta - 2cd \zeta^2 +d^2 \zeta \\[2pt]
	-c^2 \zeta  & 2c^2 \zeta^2 + 1-cd \zeta
\end{bmatrix}.
\]

So we have 
\[
\mathrm{tr}\rho(g^{g^{\alpha}} g^{-2}) 
= \pm (2c^2 \zeta^2  + 2).
\]  
To express this by $\rho(\alpha)$ 
we write 
$\rho(\alpha) 
= 
\begin{bmatrix}
	x & y \\[2pt]
	z & u
\end{bmatrix}$.

Then $\rho(\gamma) = \rho(\alpha^{-1} g \alpha) 
= 
\begin{bmatrix}
	1+zu\zeta  & u^2\zeta \\[2pt]
	 -z^2\zeta & 1-zu\zeta 
\end{bmatrix}
=
\begin{bmatrix}
	a & b \\[2pt]
	c & d
\end{bmatrix}$. 

This implies 
\[
2c^2 \zeta^2  + 2
=
2z^4 \zeta^4  + 2. 
\]

Since $\alpha$ is arbitrarily chosen, 
we expect there are infinitely many elements $\alpha \in G(K)$ such that 
$2z^4 \zeta^4 + 2$ 
takes infinitely many values up to sign. 

Recall that $g \in G(K)$ is peripheral if and only if $\mathrm{tr}\rho(g) = \pm 2$. 
Thus if $2z^4 \zeta^4 + 2 \ne \pm 2$, 
then $g^{g^{\alpha}} g^{-2}$ is not peripheral. 
\end{proof}

\medskip

\begin{example}
\label{figure-eight_holonomy}
Let $K$ be the figure-eight knot. 
Take a meridian $\mu$ and a non-trivial element $h$ as depicted in Figure~\ref{figure-eight_mu_h}. 

\begin{figure}[htb]
\centering
\includegraphics[bb=0 0 121 113,width=0.2\textwidth]{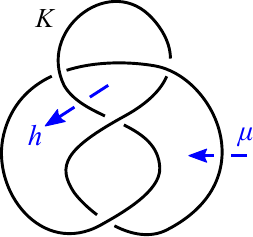}
\caption{$\mu$ and $h$ generate $G(K)$.} 
\label{figure-eight_mu_h}
\end{figure}

Let us take a holonomy representation 
$\rho \colon G(K) \to \mathrm{PSL}_2(\mathbb{C})$ so that 
$\rho(\mu) = 
\begin{bmatrix}
	1 & 1 \\[2pt]
	0 & 1
\end{bmatrix}$ and 
$\rho(h) = \begin{bmatrix}
	1 & 0 \\[2pt]
	-\omega & 1
\end{bmatrix}$, where $\omega = \frac{-1 + \sqrt{3}\,i}{2}$; see \cite{CallahanReid_IMT,Goda_IMT}. 

Then $\lambda = h\mu^{-1}h^{-1}\mu^2 h^{-1}\mu^{-1}h$, and 
$\rho(\lambda) = 
\begin{bmatrix}
	1 &2\sqrt{3}\,i \\[2pt]
	0 & 1
\end{bmatrix}$.  
For a slope element $\mu^p \lambda^q$, 
we have
$\rho(\mu^p \lambda^q) = 
\begin{bmatrix}
	1 & p + 2\sqrt{3}\,qi \\[2pt]
	0 & 1
\end{bmatrix}$. 

For simplicity put $g = \mu^p \lambda^q$. 
Then it follows from Theorem~\ref{same_trivializing_set} that 
\[
\mathcal{S}_K(g) = \mathcal{S}_K(g^{g^{h^n}} g^{-2})\quad \textrm{for any integer}\ n > 0.
\]
Now we apply Proposition~\ref{conjugacy} to see that 
$g^{g^{h^m}} g^{-2}$ is conjugate to $g^{g^{h^n}} g^{-2}$ for integers $m, n > 0$ if and only if $m = n$.

Note that 
$\rho(h^n) = \begin{bmatrix}
	1 & 0 \\[2pt]
	-n\omega & 1
\end{bmatrix}$, 
and $\zeta$ in Proposition~\ref{conjugacy}(2) is  $p + 2\sqrt{3}q i$. 
Then Proposition~\ref{conjugacy}(2) shows that 
$g^{g^{h^m}} g^{-2}$ is conjugate to $g^{g^{h^n}} g^{-2}$ if and only if 
$2 (-m\omega)^4 \zeta^4 + 2 = 2 (-n\omega)^4 \zeta^4 + 2$, i.e. $m = n$. 
\end{example}

\section{Separation Property}
\label{section:separation}

\subsection{Separation of two slopes}

Examples~\ref{torus_knot_cyclic}
and \ref{pretzel_non_separation} show that even Separation Property (i) does not hold in general. 
On the other hand, 
in Example~\ref{torus_knot_cyclic} the slope $r_0$ is a {\em finite surgery slope}, 
i.e. $\pi_1(K(r_0))$ is finite, 
and in Example~\ref{pretzel_non_separation} the slope $18$ is a finite surgery slope. 
In fact we have:

\medskip

\begin{lemma}[\cite{IMT_Magnus_IMT}]
\label{inclusion}
Let $K$ be a non-trivial knot. 
Assume that $\langle \! \langle r \rangle\!\rangle \subset \langle \! \langle s \rangle\!\rangle$. 
Then $s = r$ or $s$ is a finite surgery slope. 
\end{lemma}

\medskip

Combining this lemma with the peripheral Magnus property (Theorem~\ref{peripheral_Magnus}), 
we immediately obtain: 

\medskip

\begin{theorem}
\label{separation_two_slopes}
Let $K$ be a non-trivial knot. 
For two given distinct slopes $r$ and $s$ in $\mathbb{Q}$, 
if $s$ is not a finite surgery slope, 
then there exists an element $g \in G(K)$ such that 
$r \in \mathcal{S}_K(g)$ and $s \not\in \mathcal{S}_K(g)$.
\end{theorem}

\begin{proof}
Since $r \ne s$, by Theorem~\ref{peripheral_Magnus} $\langle\!\langle r \rangle\!\rangle \ne  \langle\!\langle s \rangle\!\rangle$. 
Furthermore, since $s$ is not a finite surgery slope,  $\langle\!\langle r \rangle\!\rangle \not\subset  \langle\!\langle s \rangle\!\rangle$.
Hence, we have a non-trivial element $g \in \langle\!\langle r \rangle\!\rangle - \langle\!\langle s \rangle\!\rangle$. 
This means that $r \in \mathcal{S}_K(g)$, but $s \not\in \mathcal{S}_K(g)$, and $g$ is a desired element. 
\end{proof}

Note that if 
$\langle\!\langle r \rangle\!\rangle \subset  \langle\!\langle s \rangle\!\rangle$ for distinct slopes $r$ and $s$, 
then $r$ is not a finite surgery slope \cite[Proposition~2.4]{IMT_Magnus_IMT}. 
So in the case where $s$ is a finite surgery slope in Theorem~\ref{separation_two_slopes}, 
we have an element $g$ such that $s \in \mathcal{S}_K(g)$, but $r \not\in \mathcal{S}_K(g)$. 

\bigskip

\subsection{Separation of two finite families of slopes}

In this subsection we will prove Theorem~\ref{separation}. 
We may apply the proof of Theorem~\ref{separation_hyperbolic}.  
Note that \cite[Claim~6.1]{IMT_realization_IMT} is true for non-trivial knots, because a satellite knot has at most one reducing surgery; 
see \cite{Sch_red_IMT}. 
(Actually if $K$ is a satellite knot admitting a reducing surgery, 
then it is a cable of a non-trivial knot and the surgery slope is the cabling slope.)
Then the argument is verbatim except for the first step. 
In the first step, we need to see that 
for the given finite subset $\mathcal{S}$ of $\mathbb{Q}$ which does not contain a Seifert surgery, 
there is an element $g$ such that $\mathcal{S}_K(g) \subset \mathbb{Q} - \mathcal{S}$. 
Recall that a finite surgery is a Seifert surgery. 
In what follows, we will prove the existence of such an element when $\mathcal{S}$ does not contain a finite surgery. 

\medskip

\begin{proposition}
\label{non-trivial}
Let $K$ be a non-trivial knot and $\mathcal{S}$ a finite subset of $\mathbb{Q}$ which does not contain a finite surgery. 
Then there are infinitely many, mutually non-conjugate elements $x$ in $[G(K), G(K)]$ 
such that $p_s(x) \ne 1$ for all slopes $s \in \mathcal{S}$. 
\end{proposition}

\medskip

To prove Proposition~\ref{non-trivial} we consider three cases according as Thurston's hyperbolization theorem: 
$K$ is a hyperbolic knot (Propositions~\ref{hyperbolic_persistent_[G,G]_noncyclic}), 
a torus knot (Proposition~\ref{persistent_[G,G]_torus}) or a satellite knot (Propositions~\ref{persistent_[G,G]_satellite} and \ref{cable_torus}). 


Since $K$ is assumed to be non-trivial, 
any cyclic surgery slope $s$ is a finite surgery slope. 
Because, if $\pi_1(K(s)) \cong \mathbb{Z}$, 
then $K$ is the unknot and $s = 0$ \cite{GabaiIII_IMT}.  


For hyperbolic knots we have already established the following result \cite{IMT_realization_IMT}. 

\medskip

\begin{proposition}[\cite{IMT_realization_IMT}]
\label{hyperbolic_persistent_[G,G]_noncyclic}
Let $K$ be a hyperbolic knot in $S^{3}$.  
Then there exist infinitely many, mutually non-conjugate elements $g \in [G(K),G(K)]$ 
such that $p_s(g) \neq 1$ in $\pi_1(K(s))$ for all non-cyclic surgery slopes $s \in \mathbb{Q}$. 
In particular, 
if $K$ has no cyclic surgery slope, then there exist infinitely many, mutually non-conjugate elements $g \in [G(K),G(K)]$ such that 
$\mathcal{S}_K(g) = \emptyset$. 
\end{proposition}

\medskip
If $K$ has a cyclic surgery slope $s$, 
then $\pi_1(K(s)) \cong G(K)/ \langle\!\langle s \rangle\!\rangle$ is abelian, 
and $[G(K),G(K)] \subset \langle\!\langle s \rangle\!\rangle$. 
Thus for every element $g$ in $[G(K),G(K)]$, 
$p_s(g) = 1$ in $\pi_1(K(s))$. 

So in the following we consider two cases: $K$ is a torus knot, or a satellite knot.  

\medskip

\begin{proposition}
\label{persistent_[G,G]_torus}
Let $K$ be a torus knot in $S^{3}$. 
Then there exist infinitely many, mutually non-conjugate elements 
$g \in [G(K),G(K)]$ such that $p_s(g) \neq 1$ for all non-cyclic surgery slopes $s \in \mathbb{Q}$. 
\end{proposition}

\begin{proof}
Recall that 
for the $(p,q)$-torus knot $K=T_{p,q}$ $(0<p<|q|)$, the knot group $G(K)$ has a presentation 
\[
\langle x,y \: | \: x^{p}=y^{q} \rangle 
 = \langle x\rangle \ast_{\langle x^{p}=y^{q} \rangle} \langle y\rangle, 
 \]
which is the amalgamated free product of two infinite cyclic groups. 
In particular, 
$G(K)$ is generated by two elements $x$ and $y$. 
Let $g = [x, y]$. 
If $p_s(g) = 1$, then $\pi_1(K(s))$ is abelian, and hence it is cyclic. 
Thus $s$ is a cyclic surgery. 
This shows that $p_s(g) \ne 1$ for any non-cyclic surgery slope $s$ of $K$. 

For $n>0$ let $w_n=y(xy)^{n+1}$ and $g_n=g^{g^{w_n}}g^{-2}$. 
Since $g \in [G(K), G(K)]$, 
so does $g_n$. 
By Corollary~\ref{same_trivialization_residually_finite}, 
$p_s(g_n)\neq 1$ for all non-cyclic surgery slopes $s$ of $K$.

Let us show that $g_n$ and $g_m$ are conjugate if and only if $n = m$.  
Note that 
\[
W_n = g^{w_n} = (y(xy)^{n+1})^{-1} g (y(xy)^{n+1})
\]
and 
\[
\begin{split}
g_n & = g^{g^{w_n}} g^{-2} = g^{W_n} g^{-2} = (W_n)^{-1} g W_n\, g^{-2} \\
& = \left( (y(xy)^{n+1})^{-1} g  (y(xy)^{n+1}) \right)^{-1} g  \left( (y(xy)^{n+1})^{-1} g  (y(xy)^{n+1}) \right)\, g^{-2}.
\end{split}
\]
Put $X=x^{-1}$, $Y=y^{-1}$ for simplicity of notation.
Then 
\[
\begin{split}
g_n &  = (y(xy)^{n+1})^{-1} [x,y]^{-1} (y(xy)^{n+1}) [x,y] (y(xy)^{n+1})^{-1}[x,y]((y(xy))^{n+1})\\
& \quad [x,y]^{-2}\\
& =
 ( (YX)^{n}YXY) (yxYX) (y(xy)^{n+1}) (xyXY)  ((YX)^{n+1}Y) (xyXY)\\
 &\quad (yxy(xy)^{n-1}xy) (yxYX)^{2}\\
& = (YX)^{n}Y^{2}X y(xy)^{n+2}  XY^{2}X(YX)^{n}Y xy^{2}(xy)^{n-1}xy^2xYXyxYX.
\end{split}
\]

The last word is cyclically reduced. 
Then it follows from \cite[Theorem~4.6]{MKS_IMT} that 
$g_n$ and $g_m$ are not conjugate. 
\end{proof}

For later convenience we note the following, 
which will be used in the proof of Proposition~\ref{cable_torus}. 

\medskip

\begin{claim}
\label{notP(K)_torus}
$g,\ g_n$ are non-peripheral in $G(K)$. 
\end{claim}

\begin{proof}
We show that $g$ is not conjugate into $\pi_1(\partial E(K))$. 
Assume for a contradiction that $h^{-1} g h \in \pi_1(\partial E(K))$ for some $h \in G(K)$. 
Then since $g \in [G(K), G(K)]$, 
$h^{-1} g h$ is null-homologous, and thus it is $\lambda_K^m$ for some integer $m$. 
Then $p_0(g) = p_0(h \lambda_K^m h^{-1}) = 1$ in $\pi_1(K(0))$. 
This implies $0$ is a cyclic surgery slope, contradicting Proposition~\ref{persistent_[G,G]_torus}. 
Hence, $g$ is non-peripheral. 
The proof for $g_n$ is identical. 
\end{proof}

\medskip

Let us turn to the proof of Proposition~\ref{non-trivial} for satellite knots. 
Note that a satellite knot other than the $(2ab \pm 1, 2)$--cable of $T_{a, b}$ does not admit a cyclic surgery \cite{Wu_IMT}.  
In particular, 
knots in Proposition~\ref{persistent_[G,G]_satellite} below has no cyclic surgeries. 
As an application of Propositions~\ref{persistent_[G,G]_torus} and \ref{hyperbolic_persistent_[G,G]_noncyclic} we may obtain: 

\medskip

\begin{proposition}
\label{persistent_[G,G]_satellite}
Let $K$ be a satellite knot which is not a $(abq \pm 1, q)$-cable of the $(a, b)$--torus knot $T_{a, b}$.  
Then there exist infinitely many, mutually non-conjugate elements $x \in [G(K),G(K)]$ such that $p_s(x) \neq 1$ for all slopes $s \in \mathbb{Q}$. 
\end{proposition}

\begin{proof}
The reason why we exclude cable knots $(abq \pm 1, q)$-cable of $T_{a, b}$ will be clarified in Case 3 in the following proof. 
(See also Remark~\ref{cable}.)

Recall that any satellite knot $K$ has a hyperbolic knot or a torus knot as a companion knot $k$. 
Let $V$ be a tubular neighborhood of $k$ containing $K$ in its interior. 
Then $E(K) = E(k) \cup (V - \mathrm{int}N(K))$ and
$G(K)= G(k) \ast_{\pi_1(\partial V)} \pi_1(V - \mathrm{int}N(K))$.

\medskip

\noindent
\textbf{Notations:}\quad
Throughout the proof we distinguish various projections as follows: 
\begin{itemize}
\item $p^{K}_s\colon G(K) \rightarrow \pi_1(K(s))$ is the projection induced from $s$--Dehn filling on $K$, 
which we simply denote by $p_s$.
\item $p^{k}_s\colon G(k) \rightarrow \pi_1(k(s))$ is the projection induced from $s$--Dehn filling on a companion knot $k$.
\end{itemize}

For a slope $s$ of $K$ we denote by $V(K; s)$ the manifold obtained from $V$ by $s$-surgery on $K \subset V$.

\medskip

Since we have already proved Proposition~\ref{non-trivial} for torus knots and hyperbolic knots, we have:

\medskip

\begin{claim}
\label{companion}
There are infinitely many, mutually non-conjugate elements 
$x \in [G(k), G(k)] \subset [G(K), G(K)]$ such that $p^{k}_s(x)\neq 1$ for all non-cyclic surgery slopes $s \in \mathbb{Q}$ of $k$. 
\end{claim}

\medskip

Following Claim~\ref{notP(K)_torus} and \cite[Claim 5.3]{IMT_realization_IMT}, 
we may assume that these elements are not conjugate into the subgroup 
$P(k)=\pi_1(\partial V)$. 
\textit{In the following we use $x$ to denote such an element.}

Let us take a slope $s\in \mathbb{Q}$.  
Since $K$ has no cyclic surgery slope, 
$s$ is not a cyclic surgery slope of $K$. 

\medskip

\noindent
\textbf{Case 1.} $\partial V(K; s)$ is incompressible.
\smallskip

Then $\pi_1(K(s))=G(k)*_{\pi_1(\partial V(K; s))}\pi_1(V(K;s))$ is an amalgamated free product. Hence $G(k)$ injects into $\pi_1(K(s))$.
Since $x$ is non-trivial in $G(k) = \pi_1(E(k))$, 
$p_s(x)=x$ is also non-trivial in $\pi_1(K(s))$

\medskip

\noindent
\textbf{Case 2.}  $V(K; s) = S^1 \times D^2$.
\smallskip

Then $K(s) = E(k) \cup V(K; s) = k(s/w^2)$, 
where $w$ is the winding number of $K$ in $V$ \cite{Go_satellite_IMT}. 
Since $\pi_1(K(s))$ is not cyclic, neither is $\pi_1(k(s/w^2))$. 
Hence $s/w^2$ is not a cyclic surgery slope of $k$, 
and $p_{s}(x)=p^{k}_{s/w^2}(x) \ne 1 \in \pi_1(k(s/w^2)) = \pi_1(K(s))$. 

\medskip 

\noindent
\textbf{Case 3.} $\partial V(K;s)$ is compressible and $V(K; s) = (S^1 \times D^2)\, \#\, W$ for some closed $3$--manifold $W \ne S^{3}$.

\smallskip

Then $K(s) = E(k) \cup V(K; s) = k(s/w^2)\, \#\, W$ \cite{Go_satellite_IMT}.  
Since $\partial V(K;s)$ is compressible, 
\cite[Corollary~2.5]{GabaiII_IMT} and \cite{Sch_JDG_IMT} show that $w \ne 0$. 
Hence following \cite{GLu_S3_IMT}, 
$k(s/w^2) \ne S^3$ and $K(s) = k(s/w^2)\, \#\, W$ is reducible. 
Then \cite[Corollary~1.4]{BZ_finite_JAMS_IMT} shows that $K$ is a $(p, q)$--cable of a torus knot $T_{a, b}$ 
for some integers $p, q, a, b$, 
and the surgery slope $s$ is the cabling slope $pq$. 
Note that $w = q \ge 2$ and 
the companion knot $k$ is $T_{a, b}$. 

Assume first that $\pi_1(k(s/w^2))$ is not cyclic. 
Then $p^k_{s/w^2}(x) \ne 1$ in $\pi_1(k(s/w^2))$, 
which injects into $\pi_1(K(s)) = \pi_1(k(s/w^2)) \ast \pi_1(W)$. 
Hence $p_{s/w^2}(x) \ne 1$ in $\pi_1(K(s))$. 

Next assume that $\displaystyle\pi_1(k(s/w^2)) = \pi_1(T_{a, b}(pq /q^2)) = \pi_1(T_{a, b}(p/q))$ is cyclic. 
This then implies that the distance between two slopes $p/q$ and $ab$ should be one, 
i.e. $|abq - p| = 1$. 
So $K$ is a $(abq \pm 1, q)$--cable of $T_{a, b}$. 
This contradicts the initial assumption.  

\medskip

Finally we show that there are infinitely many, mutually non-conjugate elements $x \in [G(K), G(K)]$ with 
$p_s(x) \ne 1$ for all slopes $s \in \mathbb{Q}$. 
In our proof the elements $x \in G(k)$ are mutually non-conjugate in $G(k)$.  
So it is sufficient to see that such elements are still non-conjugate in $G(K)$. 
Actually this immediately follows from a fact that for the amalgamated free product $G=A*_C B$ and elements $a, a' \in A - C$, 
$a$ and $a'$ are conjugate in $G$ if and only if they are conjugate in $A$ \cite[Theorem 4.6]{MKS_IMT}. 
In our setting $G = G(K), A = G(k), B = \pi_1(V(K; s))$ and $C = P(k) = \pi_1(\partial E(k))$. 
\end{proof}

\medskip

\begin{remark}
\label{cable}
If $K$ is a $(abq \pm 1, q)$-cable of $T_{a, b}$, 
then the reducing surgery on $K$ induces a cyclic surgery on the companion knot 
$k = T_{a, b}$. 
Therefore 
for any $x \in [G(T_{a, b}), G(T_{a, b})]$, 
$p_{(abq\pm 1)q}(x) = 1$ for the reducing surgery slope $(abq\pm 1)q$. 
\end{remark}

In the remaining of this section we focus on a $(abq \pm 1, q)$--cable of $k = T_{a, b}$.  
The situation described in Remark~\ref{cable} forces us to pay further attention to take desired elements in $G(K)$. 

\medskip

\begin{proposition}
\label{cable_torus}
Let $K$ be a $(abq \pm 1, q)$--cable of the $(a, b)$--torus knot $k = T_{a, b}$ \($|q| \ge 2$\). 
Let $\mathcal{S}$ be a finite subset of $\mathbb{Q}$ which does not contain a finite surgery.  
Then there are infinitely many, mutually non-conjugate elements $x$ in $[G(K), G(K)]$ 
such that $p_s(x) \ne 1$ for all slopes $s \in \mathcal{S}$. 
\end{proposition}

\begin{proof}
Decompose $E(K)$ as $E(k) \cup (V- \mathrm{int}N(k))$, 
where $E(k) = E(T_{a, b})$, and $V- \mathrm{int}N(k)$ is a $(p, q)$--cable space  ($p = abq \pm 1$), 
and $\partial E(k) = \partial V$.

\medskip

Let $\tau_a, \tau_b$ be exceptional fibers of $E(k) = E(T_{a, b})$ of indices $a, |b|$, respectively; 
we use the same symbol $\tau_a, \tau_b$ to denote the elements in $G(k)$ represented by these exceptional fibers. 
Note that $\tau_a, \tau_b$ generate $G(k)$. 
Let us take an element $g = [\tau_a, \tau_b] \in [G(k), G(k)] \subset [G(K), G(K)]$. 

Let us consider 
$g  \lambda_K^{\ell} = [\tau_a, \tau_b] \lambda_K^{\ell} \in [G(K), G(K)] \subset G(K)$
for an integer $\ell$. 

In what follows we show that there exist infinitely many integers $\ell > 0$ such that
$p_s(g \lambda_K^\ell) \neq 1$ for all slopes  $s \in \mathcal{S}$.

Following \cite{Go_satellite_IMT} we have:  

\[
V(K;s) = \begin{cases}
    \textrm{a boundary-irreducible Seifert fiber space}, & |npq-m| > 1, \\
    S^1 \times D^2, & |npq-m| = 1, \\
    S^1 \times D^2 \# L(q, p), & |npq-m| = 0.
  \end{cases}
\]

Accordingly we have:
\[
K(s)  = \begin{cases}
    E(k) \cup V(K; s), \textrm{which is a graph manifold}, & |npq-m| > 1, \\
    k(s \slash q^2), & |npq-m| = 1, \\
    k(p \slash q) \# L(q, p), & |npq-m| = 0.
  \end{cases}
\]

\medskip

\noindent
\textbf{Case 1}.  $|npq-m|>1$. 

Assume first that $ s = 0$. 
Then 
\[
p_0([\tau_a, \tau_b]  \lambda_K^{\ell}) 
= p_0([\tau_a, \tau_b]) \ne 1, 
\]
because $p_0([\tau_a, \tau_b])\in G(k) \subset G(k) \ast_{\pi_1(T)} \pi_1(V(K; 0)) = \pi_1(K(0))$, 
where $T = \partial E(k) = \partial V$. 

For any $0 \ne s \in \mathbb{Q}$, 
we have 
\[
p_s([\tau_a, \tau_b]) p_s( \lambda_K) ^{\ell}
\in G(k) \ast_{\pi_1(T)} \pi_1(V(K; 0)).
\]
Note that $p_s([\tau_a, \tau_b]) \ne 1 \in \pi_1(G(k))$. 
Note also that  
$p_s(\lambda_K) \ne 1$, because $s \ne 0$ and $s$ is not a finite surgery slope of $K$; see Lemma~\ref{inclusion} (\cite[Lemma~5.2]{IMT_realization_IMT}). 
Now suppose for a contradiction that 
$p_s([\tau_a, \tau_b]) p_s( \lambda_K) ^{\ell} = 1$. 
Then $p_s([\tau_a, \tau_b]) = p_s( \lambda_K) ^{-\ell}$. 
This means that  $p_s([\tau_a, \tau_b]) = p_s( \lambda_K) ^{-\ell} \in \pi_1(T)$. 
However, Claim~\ref{notP(K)_torus} shows that $p_s([\tau_a, \tau_b])$ is non-peripheral in $G(k)$, 
a contradiction. 
 
\medskip

\noindent
\textbf{Case 2}. $|npq-m|=1$. 

The assumptions $|npq-m|=1$ and $q \ge 2$ show that $s = m/n \ne 0$ and $s \ne pq$, 
in particular $K(s) = k(s/q^2)$ is irreducible. 

Let us take $s \in \mathcal{S}$. 
Assume that for some integer $\ell_s$ (depending on $s$) we have 
 \[
p_s([\tau_a, \tau_b]  \lambda_K^{\ell_s}) 
=  p_s([\tau_a, \tau_b]) p_s( \lambda_K)^{\ell_s}
= 1,\  \mathrm{i.e.}\ 
 p_s([\tau_a, \tau_b]) = p_s( \lambda_K^{-\ell_s})
\]
Now we show that there is at most one such an integer $\ell_s$ for each $s \in \mathcal{S}$. 
Suppose that  
\[
p_s([\tau_a, \tau_b]  \lambda_K^{\ell'_s}) 
=  p_s([\tau_a, \tau_b]) p_s( \lambda_K)^{\ell'_s}
= 1,\  \mathrm{i.e.}\ 
 p_s([\tau_a, \tau_b]) = p_s( \lambda_K)^{-\ell'_s}
\]
as well for an integer $\ell'_s$. 
Then we have 
\[
p_s( \lambda_K)^{\ell_s} =  p_s( \lambda_K)^{\ell'_s}, \ \mathrm{i.e.}\ p_s( \lambda_K)^{\ell_s- \ell'_s} = 1
\]
in $\pi_1(K(s)) = \pi_1(k(s/q^2))$. 
Since $s$ is neither $0$--slope nor a finite surgery slope (by the assumption), 
$p_s(\lambda_K) \ne 1$; see Lemma~\ref{inclusion} (\cite[Lemma~5.2]{IMT_realization_IMT}).
Furthermore, 
since $K(s) = k(s/q^2)$ is irreducible, 
its fundamental group is torsion free. 
Hence, $\ell_s = \ell'_s$. 
Then, since $\mathcal{S}$ is finite,  
we have a constant $N > 0$ such that 
$p_s([\tau_a, \tau_b]  \lambda_K^{\ell}) \ne 1$ in $\pi_1(K(s))$ for all $s \in \mathcal{S}$ whenever $\ell \ge N$.

\medskip

\noindent
\textbf{Case 3}. $|npq-m|=0$ i.e. $s = pq \ne 0$. 

In this case 
\[
K(pq) = k(p/q) \# L(q, p) = k((abq \pm 1)/q) \# L(q, p) 
= L(p, qb^2) \# L(q, p).
\]
Note that $\pi_1(K(pq)) = \pi_1(L(p, qb^2)) \ast \pi_1(L(q, p))$ is infinite. 
Since $p_s([\tau_a, \tau_b]) = 1 \in \pi_1(L(p, qb^2))$, 
we have $p_s([\tau_a, \tau_b] \lambda^{\ell}) = p_s([\tau_a, \tau_b]) p_s(\lambda_K)^{\ell} = p_s(\lambda_K)^{\ell}$, 
which is non-trivial, because $s \ne 0$ and $s$ is not a finite surgery slope \cite[Lemma~5.2]{IMT_realization_IMT}.

\medskip 

Finally we show that there are infinitely many, mutually non-conjugate elements 
$[\tau_a, \tau_b]\lambda_K^{\ell}  \in [G(K), G(K)]$.

Let $\phi \colon G(K) \to \mathbb{R}$ be a homogeneous quasimorphism of defect $D(\phi)$. 
Then it satisfies 
\[ D(\phi)=\sup_{g,h \in G(K)}|\phi(gh)-\phi(g)-\phi(h)| < \infty, \quad \phi(g^{k})=k\phi(g) \ (\forall g\in G(K), k \in \mathbb{Z}).\]
Note that homogeneous quasimorphism is constant on conjugacy classes \cite[2.2.3]{Cal_MSJ_IMT}. 

Following Bavard's Duality Theorem \cite{Bavard_IMT} we have 
\[ \mathrm{scl}_{G(K)}(\lambda_K)=\sup_{\phi} \frac{|\phi(\lambda_K)|}{2D(\phi)} \]
where $\phi\colon G(K) \rightarrow \mathbb{R}$ runs over all homogeneous quasimorphisms of $G(K)$ which are not homomorphisms. 
Since $\mathrm{scl}_{G(k)}(\lambda_K) = g(K) -1/2 > 0$ \cite[Proposition 4.4]{Cal_MSJ_IMT}, 
if necessary by taking $\phi' = - \phi$, 
we may take $\phi$ so that $\phi(\lambda_K) > 0$.

Then we have an inequality 
\[
D(\phi) \ge |\phi([\tau_a, \tau_b] \lambda_K^{\ell}) - \phi ([\tau_a, \tau_b]) - \phi(\lambda_K^{\ell})|, 
\]
which implies 
\[
\phi([\tau_a, \tau_b] \lambda_K^{\ell}) 
\ge \phi([\tau_a, \tau_b]) +
 \phi(\lambda_K^{\ell})  - D(\phi)
= \phi([\tau_a, \tau_b]) +
 \ell \phi(\lambda_K)  - D(\phi).
\]

Hence $\displaystyle \lim_{\ell \to \infty} \phi([\tau_a, \tau_b] \lambda_K^{\ell})  \to \infty$. 
Since a homogeneous quasimorphism $\phi$ is conjugation invariant, 
this shows that $\{[\tau_a, \tau_b] \lambda_K^{\ell} \}_{\ell \in \mathbb{Z}}$ has infinitely many, 
mutually non-conjugate elements.
\end{proof}

\medskip

\section{Non-separable pairs of finite families of slopes}
\label{example}

In this section we provide some non-separable pairs of finite families of slopes. 

\medskip

\begin{example}
\label{torus_knot_pq}
Let K be a torus knot $T_{p,q}$. Then the slope $pq$ represented by a regular fiber has a distinguished property. 
\[
\langle \! \langle pq \rangle\!\rangle \cap \langle \! \langle r \rangle\!\rangle \cong \begin{cases} \mathbb{Z} & \mbox{if}\ r \mbox{ is a finite surgery slope,} \\ \{1\} & \mbox{if $r\ne pq$ and $r$ is not a finite surgery slope.}\end{cases}
\] 
See \cite[Proposition 5.4]{IMT_Magnus_IMT}. 
Thus if $pq \in \mathcal{S}_{K}(g)$, 
then $r \not \in \mathcal{S}_{K}(g)$ for all non-finite surgery slopes $r \ne pq$. 
So for any non-finite surgery slope $r$,  there is no non-trivial element $g \in G(K)$ such that 
$\mathcal{R} = \{ pq, r \} \subset \mathcal{S}_K(g)$. 
In particular, for any $\mathcal{S} \subset \mathbb{Q}$ with $\mathcal{R} \cap \mathcal{S} = \emptyset$, 
there is no element $g \in G(K)$ such that $\mathcal{R} \subset \mathcal{S}_K(g) \subset \mathbb{Q} -  \mathcal{S}$. 
\end{example}

\medskip

In general,  
an inclusion $\langle \! \langle r \rangle\!\rangle \subset \langle \! \langle s \rangle\!\rangle$ forces us the restriction that 
 $r \in \mathcal{S}_K(g)$ implies $s \in \mathcal{S}_K(g)$.  
This gives the following examples. 

\medskip

\begin{example}
\label{torus_knot_cyclic}
Let $K$ be a torus knot $T_{p, q}$ $(p > q \ge 2)$. 
Then for each finite surgery slope $r_0 \in \mathbb{Q}$, 
\cite[Theorem~6.4]{IMT_Magnus_IMT} shows that there is an infinite descending chain 
\[
\langle \! \langle r_0 \rangle\!\rangle \supset \langle \! \langle r_1 \rangle\!\rangle \supset \langle \! \langle r_2 \rangle\!\rangle \supset \cdots.
\]
Hence 
if $g \in \langle \! \langle r_n \rangle\!\rangle$, then $g \in \langle \! \langle r_m \rangle\!\rangle$ for any pair $m, n$ with $m < n$. 
This means that there is no element $g \in G(K)$ such that 
$r_n  \in \mathcal{S}_K(g)$ and $r_m \not\in \mathcal{S}_K(g)$. 
\end{example}

\medskip

Even for hyperbolic knots, we have: 

\medskip

\begin{example}
\label{pretzel_non_separation}
Let $K$ be the $(-2, 3, 7)$--pretzel knot. 
Choose two slopes $\frac{18}{5}$ and $18$. 
Then as shown in \cite[Example 6.2]{IMT_Magnus_IMT}, 
$\langle \! \langle \frac{18}{5} \rangle\!\rangle \subset \langle \! \langle 18 \rangle\!\rangle$. 
Thus there is no element $g$ such that 
$\frac{18}{5} \in \mathcal{S}_K(g)$ and 
$18  \not\in \mathcal{S}_K(g)$. 
\end{example}

\medskip

Therefore the set $\mathcal{S}_K(g)$ is not arbitrary. 

\bigskip

\textbf{Acknowledgments} 
We would like to thank the referee for careful reading and useful comments.

\end{document}